\documentclass[10pt,twoside,leqno]{amsart}
\usepackage{amsmath,amssymb,amsfonts,amsthm,enumerate}
\usepackage{srcltx}
\author{A.~Lengyeln\'e T\'oth} 
\address{University of Debrecen, 4010 Debrecen, Pf. 12}
\email{totha@klte-gyakorlo.sulinet.hu}
\author{Z.~Kov\'acs}
\address{College of Ny\'iregyh\'aza,  4400 Ny\'iregyh\'aza, S\'ost\'oi \'ut 31/b}
\email{kovacsz@nyf.hu}
\thanks{The second author was supported by the Japanase-Hungarian bilateral project Nr.\ T\'ET\_10-1-2011-0065}
\title{Left invariant Randers metrics on the 3-dimensional Heisenberg group}
\date{\today}
\dedicatory{Dedicated to Professor Lajos Tam\'assy on the occassion of his ninetieth birthday}
\newtheorem{theorem}{Theorem}[section]
\newtheorem*{theoremm}{Theorem}
\newtheorem{proposition}[theorem]{Proposition}
\newtheorem{corollary}[theorem]{Corollary}
\newtheorem{lemma}[theorem]{Lemma}

\newtheorem{algorithm}[theorem]{Algorithm}
\theoremstyle{definition}
\newtheorem{definition}[theorem]{Definition}
\newcommand{\scalarprod}[2]{\left\langle{#1},{#2}\right\rangle}
\newcommand{\cartanb}[4]{\mathcal{C}_{#4}(#1,#2,#3)}
\newcommand{\cartant}[3]{\mathcal{C}^2_{#3}(#1,#2)}
\newcommand{\vf}{\ensuremath\mathfrak{X}}

\newcommand{\centrumort}{\ensuremath\mathcal{V}}
\newcommand{\centrum}{\ensuremath \mathcal{Z}}
\DeclareMathOperator{\Span}{span}
 
\begin{document}
\begin{abstract}
In the present paper we give a complete description of the Chern--Rund connection defined by a left invariant Randers metric on the 3 dimensional Heisenberg group.
\end{abstract}
\maketitle
\section{Introduction}
Randers metric is a Finsler metric which is defined as
the sum of a Riemannian metric and a 1-form. It is an object that shows strong non-Riemannian characters. The history of Randers metric goes back to G.~Randers'  research on general relativity \cite{MR0003371}.  Since then it has been widely applied in many areas, including electron optics and biology. (A more detailed account can be found in \cite{MR1273129}.) Randers metric can be naturally deduced as the solution of the famous Zermelo navigation problem \cite{MR2106471}. 

In chapter 11 of \cite{MR1747675} the authors give six reasons to study Randers metric. Number 5 is that Randers metrics are computable and this may lead to a better understanding of Finsler metrics. Our strategy in this paper was the same, we specialized our original problem of left invariant Finsler metrics on two-step nilpotent groups (presented in \cite{MR2430243}) to Randers metric on Heisenberg group. The straight motivation of the original study was P.~Eberlein's comprehensive work \cite{MR1296558} for the Riemannian case. 

In the previous paper \cite{MR2430243} we computed some geometric quantities such as curvature and flag curvature for a general left invariant Finsler metric on a two-step nilpotent group. In the first step we gave an explicit formula for the Chern--Rund connection. That paper had limitations, the reference vector for the Chern--Rund connection was chosen from the center of the respective Lie algebra. In the present paper we give a complete description of the Chern--Rund connection defined by a left invariant Randers metric on the 3-dimensional Heisenberg group. The Randers perturbation vector lies in the center of the Lie algebra in this paper.

\section{Conventions}
\subsection{Finsler metrics and the Chern-Rund connection}\label{ss:zerg5}
Through this paper we use \cite{MR1747675} as a basic reference for foundations of Finsler geometry. We consider metric structures on a differentiable manifold $N$ and 'differentiable' means $C^\infty$-differentiable. The module of tangent vector fields over $N$ is denoted by $\vf(N)$.
\begin{definition}
A \emph{Finsler manifold} $(N,F)$ is a differentiable manifold $N$ equipped with a Finsler metric $F$. A \emph{Finsler metric} on $N$ is a continuous map, $F\colon TN\to \mathbf{R}$ differentiable outside the zero section and satisfying three conditions:
\begin{enumerate}
\item $F$ is positively homogeneous,
\item if $F(X)=0$ then $X=0$,
\item $F$ is strong convex.
\end{enumerate}
\end{definition}

In the sequel we fix a nowhere vanishing vector field $W\in\vf(N)$, the so called reference vector field. Generally such a vector field does not exist globally and we arrange that all objects live on an open subset $\mathcal{U}\subset N$, where the reference vector field exists.
\begin{definition}
The \emph{osculating Riemann metric} $\scalarprod{}{}_W$ is determined by the Finslerian fundamental function $F$ and by the reference vector field $W\in\vf(N)$ in the following way:
\begin{equation}
\scalarprod{X_p}{Y_p}_W=\frac{1}{2}
\left.
\frac{\partial^2F^2(W_p+sX_p+tY_p)}{\partial s\partial t}
\right|_{s,t=0},
\quad p\in N,\ X,Y\in\vf(N).
\end{equation}
\end{definition}
\begin{definition}
For $X,Y,Z\in\vf(N)$,
\[
\cartanb{X_p}{Y_p}{Z_p}W
=\left.\frac{1}{4}\frac{\partial^3}{\partial r\partial s\partial t}
\right|_{r,s,t=0}F^2
\left(
W_p+rX_p+sY_p+tZ_p
\right)
\]
is the  \emph{(osculating) Cartan tensor}. Its $(1,2)$-type version is defined by
\[
\mathcal C^2_W\colon\vf(N)\times \vf(N)\to\vf(N),\
\scalarprod{\cartant{X}{Y}{W}}{Z}_W=\cartanb{X}{Y}{Z}{W}.
\]
\end{definition}
For the Cartan tensor we have
\begin{equation}
\cartanb{W}{X}{Y}W=\cartanb{X}{W}{Y}W=\cartanb{X}{Y}{W}W=0.
\label{eq:chern00}
\end{equation}
\begin{theorem}[\cite{MR2132661}]
The Chern--Rund connection 
$\nabla^W\colon \vf(N)\times\vf(N)\to\vf(N)$ w.r.t.~the reference vector field $W$ 
satisfies
\begin{equation}
\begin{split}
2\scalarprod{\nabla^W_XY}{Z}_W&=
X\scalarprod{Y}{Z}_W+
Y\scalarprod{Z}{X}_W-
Z\scalarprod{X}{Y}_W+\\
&\quad +\scalarprod{[X,Y]}{Z}_W-
\scalarprod{[Y,Z]}{X}_W+
\scalarprod{[Z,X]}{Y}_W-\\
&\quad-2\cartanb{\nabla^W_XW}{Y}{Z}W-
2\cartanb{\nabla^W_YW}{Z}{X}W+\\
&\quad+2\cartanb{\nabla^W_ZW}{X}{Y}W.
\end{split}
\label{eq:CR}
\end{equation}
\end{theorem}

The Chern--Rund connection is torsion-free, that is,
\begin{equation}\label{eq:chern01}
\nabla^W_XY-\nabla^W_YX-[X,Y]=0,
\end{equation}
and almost metric, that is,
\begin{equation*}
X\scalarprod{Y}{Z}_W=\scalarprod{\nabla^W_XY}{Z}_W+\scalarprod{Y}{\nabla^W_XZ}_W+
2\cartanb{\nabla^W_XW}{Y}{Z}W.
\end{equation*}

In order to get all the local components of the Chern--Rund connection w.r.t.\ a local base, it is sufficient to show that it can be eliminated from the right hand side of \eqref{eq:CR}. We can do it with the following simple algorithm.
\begin{algorithm}[`Local strategy']\label{algorithm}
Let $(E_i)$ be an orthonormal base w.r.t.\ $\scalarprod{}{}_W$.
\begin{enumerate}
\item[1.] Choose $X,Y\in\{W,E_i\}$ such that all the terms in the right hand side of \eqref{eq:CR} are explicitly known while computing for $\scalarprod{\nabla^W_XY}{E_i}_W$.
\item[2.] Set
\begin{equation*}
\nabla^W_XY=\sum_{i}\scalarprod{\nabla^W_XY}{E_i}_WE_i.
\end{equation*}
\item[3.] Repeat the previous steps until all the local components of the Chern--Rund connection are known. 
\end{enumerate}
\end{algorithm}
We give further details. 
For the first six terms of the right hand side of \eqref{eq:CR} we use the abbreviation $\mathcal A_W(X,Y,Z)$. In these terms the Chern-Rund connection does not occur.
\begin{enumerate}
\item[1a.]
Considering \eqref{eq:chern00}, equation \eqref{eq:CR} implies that
\[
2\scalarprod{\nabla^W_WW}{E_i}_W=\mathcal A_W(W,W,E_i),
\]
i.e.\ $\nabla^W_WW$ is explicitly known:
\[
2\nabla^W_WW=\sum_i\mathcal A_W(W,W,E_i)E_i.
\]
\item[1b.] Let $S\in \{E_i\}$. From  equation \eqref{eq:CR} we have
\begin{equation}\label{1b}
2\scalarprod{\nabla^W_SW}{E_i}_W=
\mathcal A_W(S,W,E_i)-
2\cartanb{\nabla^W_WW}{E_i}{S}W.
\end{equation}
Here $\nabla^W_WW$ is known from the previous step, and we get $\nabla^W_SW$.
\item[1c.] Let $S,T\in\{E_i\}$.
\begin{equation}\label{1c}
\begin{split}
2\scalarprod{\nabla^W_ST}{E_i}_W&=
\mathcal A_W(S,T,E_i)\\
&\quad-2\cartanb{\nabla^W_SW}{T}{E_i}W-
2\cartanb{\nabla^W_TW}{E_i}{S}W+\\
&\quad+2\cartanb{\nabla^W_{E_i}W}{S}{T}W.
\end{split}
\end{equation}
Here all the terms in the right hand side are known from 1b.
\end{enumerate}
\subsection{Left invariant Randers metrics on 3-dimensional Heisenberg group}

\begin{definition}\label{def:3wjswg}
Let $\mathcal Z=\operatorname{span}{Z}$ be a 1-dimensional vector space spanned by the element $Z$. Let $(X,Y)$ be any basis of $\mathbf{R}^2$.  Define $[X,Y]=-[Y,X]=Z$ with all other brackets zero. The Lie algebra $\mathcal N=\mathcal Z\oplus\mathbf{R}^2$ is the 3-dimensional \emph{Heisenberg algebra}. 
Moreover, let $\scalarprod{}{}$ denote the positive definite inner product on $\mathcal N$ for which $(X,Y,Z)$ is an orthonormal base. Thus $\operatorname{span}(X,Y)$ is the orthogonal complement of $\mathcal Z$ for which we use the notation $\mathcal Z^\perp$. 
\end{definition}

Let $\{N,\scalarprod{}{}\}$ denote the three-dimensional Heisenberg group, i.e.\ $N$ is a simply connected 2-step nilpotent group with Lie algebra $\mathcal N$ and $\scalarprod{}{}$ is the left invariant Riemannian metric  induced by left translations from the original metric given on $\mathcal N$.
In this paper we shall regard the elements of $\mathcal N$ as left invariant vector fields on $N$ determined by their values at the identity of $N$. 
We remark that the first three terms of the right hand side of \eqref{eq:CR} vanish for left invariant vector fields.

Left invariant Cartan tensor and Chern--Rund connection can be derived from a left invariant Finsler metric. To be more precise let $W\in\mathcal N$ and we may regard $\nabla^W$ as a bilinear mapping from $\mathcal N\times\mathcal N\to\mathcal N$. Similarly  the trilinear function $\mathcal C_W$ lives on $\mathcal N$, too:
\[
\mathcal C_W\colon\mathcal N\times\mathcal N\times\mathcal N\to\mathbf R
\]

It is well-known that for $X_0\in\mathcal N $ with property $\|X_0\|<1$ the function
\begin{equation}\label{eq:83gcf}
f\colon\mathcal N\to\mathbf R,\ X\mapsto f(X)=\sqrt{\scalarprod{X}{X}}+\scalarprod{X_0}{X}
\end{equation}
defines a Minkowski functional on $\mathcal N$, therefore it can be extended to a \emph{left-invariant Randers type Finsler metric} $F$ on the Lie group $N$ by left translations. Excluding the case $X_0=0$, the remaining Randers metrics are non-Riemannian \cite[p.~283]{MR1747675}. By a direct computation we can express $\scalarprod{}{}_W$ and $\cartanb{}{}{}W$ in terms of the Riemannian metric.
\begin{proposition}[\cite{MR2214213}, \cite{MR2598183}]\label{prop:7ssu}
Let $W\in\mathcal N$ and $\scalarprod{W}{W}=1$. Then
\begin{multline}\label{eq:scalar}
\scalarprod{U}{V}_W=\scalarprod{U}{V}+\scalarprod{X_0}{U}\scalarprod{X_0}{V}
-\scalarprod{X_0}{W}\scalarprod{W}{U}\scalarprod{W}{V}
\\
+
\scalarprod{X_0}{U}\scalarprod{W}{V}+
\scalarprod{X_0}{W}\scalarprod{U}{V}+
\scalarprod{X_0}{V}\scalarprod{W}{U}
\end{multline}
and
\begin{multline}\label{eq:cartan}
\cartanb{U}{V}{X}{W}=\frac{1}{2}\sum_{[U,V,X]}\left\{
\scalarprod{X_0}{W}\scalarprod{W}{U}\scalarprod{W}{V}\scalarprod{W}{X}\right.\\
-\left.\scalarprod{X_0}{W}\scalarprod{X}{V}\scalarprod{U}{W}-
\scalarprod{X_0}{X}\scalarprod{W}{V}\scalarprod{W}{U}+
\scalarprod{X_0}{U}\scalarprod{X}{V}
\right\}.
\end{multline}
where $\sum_{[U,V,X]}$ refers to the cyclic sum with respect to $U,V,X$.
\end{proposition}

\section{Determination of the Chern--Rund connection}

In the present paper we assume that
the Randers-type Minkowski functional \eqref{eq:83gcf} on the three-dimensional Heisenberg algebra ($=\operatorname{span}(X,Y,Z)$, as in definition \ref{def:3wjswg}) is determined by $X_0=\xi Z\in\mathcal Z$, ($0<\xi<1$) i.e.\ it is distinguished algebraically by the one-dimensional center of the Lie algebra. The reference vector $W$ is supposed to be normalized w.r.t.~$\scalarprod{}{}$ in the sequel.

We use the so called Berwald-Mo\'or frame (\cite{MR0345033, MR0090088}) for computation. 
\subsection{The Berwald-Mo\'or frame}\sloppy
A.~Mo\'or used in the paper \cite{MR0090088} a special orthonormal frame which was a generalization of the  Berwald frame of two-dimensional Finsler spaces.
We adapt the original definition to our context.
The first base vector is the normalized reference vector $W$: 
\[
E_1=\frac{1}{\sqrt{\scalarprod{W}{W}_W}}W.
\]
The second base vector is the normalized Cartan vector.
\begin{definition}[c.f.~\cite{MR1936741}]
The \emph{Cartan vector w.r.t.\ $W$} is the unique vector $C_W\in\mathcal N$ such that
\begin{equation}
\forall S\in\mathcal N: \scalarprod{S}{C_W}_W=
(\operatorname{trace}\mathcal C^2_W)(S)
=
\operatorname{trace}(U\mapsto \mathcal C^2_W(S,U)).
\end{equation}
\end{definition}

It follows directly from the definition  that 
\[
\scalarprod{W}{C_W}_W
=\operatorname{trace}(U\mapsto \mathcal C^2_W(W,U))
=\operatorname{trace}(U\mapsto 0)=0,
\]
i.e.\ The Cartan vector w.r.t\ $W$ is always orthogonal to $W$. Deicke's classical theorem states that $\forall W: C_W=0$ if and only if the metric is Euclidean \cite{MR0055026}. However, $C_W=0$ is possible for some $W$ in the non-Euclidean case. 
\begin{proposition} If $W\notin\mathcal{Z}$ then $C_W\not = 0$.
\end{proposition}
\begin{proof}
Let $(X_1=W,X_2,X_3)$ be an orthonormal base w.r.t.~$\scalarprod{}{}$, $g_{ij}=\scalarprod{X_i}{X_j}_W$, $(g^{ij})=(g_{ij})^{-1}$ and $S=\sum_{l=1}^3\scalarprod{S}{X_l}X_l=g^{kl}\scalarprod{S}{X_k}_WX_l$ an arbitrary vector. From the definition of the trace operator it follows that
\begin{align}\label{eq:ei64g}
\operatorname{trace}\{U\mapsto \mathcal C^2_W(S,U)\}&=
\sum_{i=1}^3\scalarprod{X_i}{\mathcal C^2_W(S,X_i)}=
g^{ij}\scalarprod{\mathcal C^2_W(S,X_i)}{X_j}_W\\
&=g^{ij}\cartanb{S}{X_i}{X_j}W.\nonumber
\end{align}
Equation \eqref{eq:cartan} gives
\begin{gather}\label{eq:iue6fg}
\cartanb{X_2}{X_2}{X_2}W=\frac{3}{2}\scalarprod{X_0}{X_2},\
\cartanb{X_2}{X_2}{X_3}W=\frac{1}{2}\scalarprod{X_0}{X_3},\\
\cartanb{X_3}{X_3}{X_2}W=\frac{1}{2}\scalarprod{X_0}{X_2},\
\cartanb{X_3}{X_3}{X_3}W=\frac{3}{2}\scalarprod{X_0}{X_3},\nonumber
\end{gather}
and all the terms $\cartanb{X_1}{X_i}{X_j}W$ vanish.
Substituting \eqref{eq:iue6fg} into \eqref{eq:ei64g} we have
\begin{equation}\label{eq:85hf8}
(\operatorname{trace}C_W^2)(S)=(g^{22}+g^{33})
\left(
\scalarprod{X_0}{S}-\scalarprod{X_0}{W}\scalarprod{S}{W}
\right).
\end{equation}
Substitute $Z$ for $S$:
\begin{align}\label{eq:ufdbdi}
\scalarprod{C_W}{Z}_W&=(\operatorname{trace}\mathcal C_W^2)(Z)\\
&=(g^{22}+g^{33})\left(
\scalarprod{X_0}{Z}-\scalarprod{X_0}{W}\scalarprod{Z}{W}\right)\nonumber\\
&=(g^{22}+g^{33})\xi(1-\scalarprod{Z}{W}^2).\nonumber
\end{align}
$g^{22}, g^{33}>0$ because $(g^{ij})$ is positive definite. $\xi\neq 0$ because the space is non-Riemannian. From the Cauchy-Schwarz inequality we have $1-\scalarprod{Z}{W}^2\geq0$, and equality holds if and only if $W=\pm Z$.
\end{proof}

In the paper \cite{MR2430243} the case of $W\in\mathcal{Z}$ was completely described. The following proposition shows that this case has a Riemannian flavour for the Randers metric.
\begin{proposition}
\begin{equation}\label{eq:8wgdfv}
\begin{tabular}{lll}
$\scalarprod{X}{X}_Z=\xi+1$,&
$\scalarprod{X}{Y}_Z=0$,&
$\scalarprod{Y}{Y}_Z=\xi+1$,
\\[1em]
$\scalarprod{X}{Z}_Z=0$,&
$\scalarprod{Y}{Z}_Z=0$,&
$\scalarprod{Z}{Z}_Z=(1+\xi)^2$,
\end{tabular}
\end{equation}
and all the local components of the Cartan tensor $\mathcal C_Z$ are zero.
\end{proposition}

The next Proposition specializes Proposition 8 of \cite{MR2430243} for the Randers metric.
\begin{proposition}\label{prop2542}
The local components of the Chern--Rund connection $\nabla^Z$ w.r.t.\ base $(X,Y,Z)$ are
\begin{align*}
&\nabla^Z_{X}X=0,\quad
\nabla^Z_{X}Y=\frac{1}{2}Z,\quad
\nabla^Z_{Y}X=-\frac{1}{2}Z,\quad
\nabla^Z_YY=0,\\
&\nabla^Z_{Z}X=\nabla^Z_{X}Z=-\frac{\xi+1}{2}Y,\quad
\nabla^Z_{Z}Y=\nabla^Z_{Y}Z=\frac{\xi+1}{2}X,\\
&\nabla^Z_{Z}Z=0.
\end{align*} 
\end{proposition}
 
From here  we suppose that $W\notin\mathcal{Z}$ and the second base vector is given by
\[
E_2=\frac{1}{\sqrt{\scalarprod{C_W}{C_W}_W}}C_W.
\]
$E_3$ completes $(E_1,E_2)$ such that $(E_1,E_2,E_3)$ is orthonormal w.r.t.\ $\scalarprod{}{}_W$. (Later we fix the orientation of the triplet.)

\begin{lemma}\label{lem:s722}
If the Randers-type Minkowski functional on the three-dimensional Heisenberg algebra is determined by $X_0=\xi Z\in\mathcal Z$, then 
$[C_W,W]=0$.
\end{lemma}
\begin{proof}
It is enough to see that if a vector $S$ satisfies $\scalarprod{S}{C_W}_W=0$ and $\scalarprod{S}{W}_W=0$ then $\scalarprod{S}{X_0}_W=0$, i.e.\
$C_W\in\operatorname{span}(W,Z)$.

We prove that if $\scalarprod{S}{C_W}_W=0$ and $\scalarprod{S}{W}_W=0$ then $\scalarprod{S}{X_0}=0$. 
By \eqref{eq:scalar}
\begin{equation}\label{eq:4gbfo}
\scalarprod{S}{W}_W=
(1+\scalarprod{X_0}{W})(\scalarprod{S}{W}+\scalarprod{X_0}{S}).
\end{equation}
Using the Cauchy-Schwarz inequality
\begin{equation}\label{cs}
\scalarprod{X_0}{W}^2\leq\scalarprod{X_0}{X_0}\scalarprod{W}{W}=\xi^2<1,
\end{equation}
from which it follows that $1+\scalarprod{X_0}{W}\neq 0$.

By \eqref{eq:85hf8} and \eqref{eq:4gbfo} conditions $\scalarprod{S}{C_W}_W=\scalarprod{S}{W}_W=0$ imply that
\begin{align}
\scalarprod{X_0}{S}-\scalarprod{X_0}{W}\scalarprod{S}{W}&=0\\
\scalarprod{X_0}{S}+\scalarprod{S}{W}&=0,\label{eq:76f7}
\end{align}
from which it follows
\[
\scalarprod{S}{W}(1+\scalarprod{X_0}{W})=0.
\]
Again, by \eqref{cs} we have 
\begin{equation}\label{eq:i6fg6}
\scalarprod{S}{W}=0,
\end{equation}
and \eqref{eq:76f7} implies that 
\begin{equation}\label{eq:8whdjh}
\scalarprod{X_0}{S}=0.
\end{equation}
Substituting \eqref{eq:i6fg6} and \eqref{eq:8whdjh} into \eqref{eq:scalar} we have
$\scalarprod{S}{X_0}_W=0$.
\end{proof}

The proof of the following statement has already been shown previously, but we formulate the result separately for future reference. 
\begin{corollary}\label{cor:7f65dgb}
If a vector $S$ satisfies $\scalarprod{S}{C_W}_W=0$ and $\scalarprod{S}{W}_W=0$ then $\scalarprod{S}{X_0}_W=0$ and $\scalarprod{S}{X_0}=\scalarprod{S}{W}=0$. In particular,
\begin{equation}\label{rq:e39hx}
\scalarprod{E_3}{X_0}_W=0 \text{ and } 
\scalarprod{E_3}{X_0}=\scalarprod{E_3}{W}=0.
\end{equation} 
\end{corollary}
\subsection{The case of $W\notin\mathcal Z$}
To compute the Cartan tensor, we require a simple technical lemma.
\begin{lemma}If the Randers-type Minkowski functional on the three-dimensional Heisenberg algebra is determined by $X_0=\xi Z\in\mathcal Z$, then
for the Berwald--Mo\'or frame we have
\begin{gather}
\scalarprod{W}{E_2}+\scalarprod{X_0}{E_2}=0\label{eq:8sf6s}\\
\scalarprod{E_3}{E_2}=0\label{eq:762f7s}\\
\left(
1+\scalarprod{X_0}{W}
\right)
\left(
\scalarprod{E_2}{E_2}-\scalarprod{X_0}{E_2}^2\right)
=1\label{eq:og4xh}\\
\scalarprod{E_3}{E_3}(1+\scalarprod{X_0}{W})=1.
\label{eq:8v45vb}
\end{gather}
\end{lemma}
\begin{proof}
All statements follow directly from  Proposition~\ref{prop:7ssu}. In more detail, using \eqref{eq:scalar} we have
\begin{align*}
0=\scalarprod{W}{E_2}_W&
=\scalarprod{W}{E_2}+\scalarprod{X_0}{W}\scalarprod{X_0}{E_2}
-\scalarprod{X_0}{W}\scalarprod{W}{W}\scalarprod{W}{E_2}\\
&\quad+\scalarprod{X_0}{W}\scalarprod{W}{E_2}+
\scalarprod{X_0}{W}\scalarprod{W}{E_2}
+\scalarprod{X_0}{E_2}\scalarprod{W}{W}\\
&=(1+\scalarprod{X_0}{W})
(\scalarprod{W}{E_2}+\scalarprod{X_0}{E_2})
\end{align*}
by the fact that $\scalarprod{W}{W}=1$.
 $1+\scalarprod{X_0}{W}\neq 0$ by \eqref{cs}, which gives \eqref{eq:8sf6s}.

Similarly,
\begin{align*}
0=\scalarprod{E_3}{E_2}_W&=\scalarprod{E_3}{E_2}+
\scalarprod{X_0}{E_3}\scalarprod{X_0}{E_2}
-\scalarprod{X_0}{W}\scalarprod{W}{E_3}\scalarprod{W}{E_2}\\
&\quad+\scalarprod{X_0}{E_3}\scalarprod{W}{E_2}
+\scalarprod{X_0}{W}\scalarprod{E_3}{E_2}+
\scalarprod{X_0}{E_2}\scalarprod{W}{E_3},\\
\intertext{but since $\scalarprod{X_0}{E_3}=0$, and $\scalarprod{W}{E_3}=0$ 
(cf.~Corollary~\ref{cor:7f65dgb})}
&=(1+\scalarprod{X_0}{W})\scalarprod{E_3}{E_2}.
\end{align*}
which, since $1+\scalarprod{X_0}{W}\neq 0$, yields \eqref{eq:762f7s}.

We next prove \eqref{eq:og4xh}. 
\begin{align*}
1=\scalarprod{E_2}{E_2}_W
&=\scalarprod{E_2}{E_2}+\scalarprod{X_0}{E_2}^2
-\scalarprod{X_0}{W}\scalarprod{W}{E_2}^2\\
&\quad +\scalarprod{X_0}{E_2}\scalarprod{W}{E_2}+
\scalarprod{X_0}{W}\scalarprod{E_2}{E_2}+
\scalarprod{X_0}{E_2}\scalarprod{W}{E_2},\\
\intertext{because of $\scalarprod{X_0}{E_2}=-\scalarprod{W}{E_2}$}
&=\scalarprod{E_2}{E_2}-\scalarprod{X_0}{E_2}^2
-\scalarprod{X_0}{W}\scalarprod{X_0}{E_2}^2
+\scalarprod{X_0}{W}\scalarprod{E_2}{E_2}\\
&=(1+\scalarprod{X_0}{W})
\left(\scalarprod{E_2}{E_2}-\scalarprod{X_0}{E_2}^2\right)
\end{align*}

Finally, we prove \eqref{eq:8v45vb}.
\begin{align*}
1=\scalarprod{E_3}{E_3}_W&
=\scalarprod{E_3}{E_3}+\scalarprod{X_0}{E_3}^2
-\scalarprod{X_0}{W}\scalarprod{W}{E_3}^2\\
&\quad+\scalarprod{X_0}{E_3}\scalarprod{W}{E_3}+
\scalarprod{X_0}{W}\scalarprod{E_3}{E_3}
+\scalarprod{X_0}{E_3}\scalarprod{W}{E_3},\\
\intertext{but since $\scalarprod{X_0}{E_3}=0$, (cf.~\eqref{rq:e39hx})}
&=\scalarprod{E_3}{E_3}\left(1+\scalarprod{X_0}{W}\right).\qedhere
\end{align*}
\end{proof}
\begin{corollary}\label{cor:76rvg}
With the notations and hypotheses above,
\begin{equation}\label{eq:9fedhc}
\scalarprod{Z}{E_2}^2=\frac{\xi^2-(w-1)^2}{\xi^2w^3},\
\scalarprod{E_2}{E_2}=\frac{\xi^2+2w-1}{w^3},
\end{equation}
where $w=\|W\|_W=1+\scalarprod{X_0}{W}$. 
Moreover, $\scalarprod{Z}{E_2}>0$.
\end{corollary}
\begin{proof}
From \eqref{eq:scalar} we get $\scalarprod{W}{W}_W^2=(1+\scalarprod{X_0}{W})^2$. Since \eqref{eq:8v45vb}, $1+\scalarprod{X_0}{W}>0$, and we proved that $w=1+\scalarprod{X_0}{W}$.

Substituting $Z$ and $C_W$ for $S$ in \eqref{eq:85hf8} we get
\begin{align*}
\scalarprod{C_W}{Z}_W&=\left(g^{22}+g^{33}\right)\left(
\scalarprod{X_0}{Z}-\scalarprod{X_0}{W}\scalarprod{Z}{W}
\right)\\
&=\left(g^{22}+g^{33}\right)\xi\left(1-\scalarprod{Z}{W}^2\right)
\end{align*}
and
\begin{align*}
\scalarprod{C_W}{C_W}_W&=\left(g^{22}+g^{33}\right)\left(
\scalarprod{X_0}{C_W}-\scalarprod{X_0}{W}\scalarprod{C_W}{W}
\right)\\
&=\left(g^{22}+g^{33}\right)\scalarprod{X_0}{C_W}(1+\scalarprod{X_0}{W}),
\quad\text{by \eqref{rq:e39hx}}.
\end{align*}
It follows that
\[
\frac{\scalarprod{C_W}{Z}_W\scalarprod{X_0}{C_W}}{\scalarprod{C_W}{C_W}_W}=
\frac{\xi\left(1-\scalarprod{Z}{W}^2\right)}{w};
\]
i.e.
\[
\scalarprod{E_2}{Z}_W
\scalarprod{E_2}{X_0}=
\frac{\xi\left(1-\scalarprod{Z}{W}^2\right)}{w}.
\]
Applying \eqref{eq:scalar} again, we obtain
\[
\scalarprod{Z}{E_2}_W
=w^2\scalarprod{Z}{E_2}=\frac{w^2}{\xi}\scalarprod{X_0}{E_2}.
\]
Thus
\[
\scalarprod{E_2}{X_0}^2
=\frac{\xi^2(1-\scalarprod{Z}{W}^2)}{w^3}=
\frac{\xi^2-(w-1)^2}{w^3}.
\]
The second statement is a straightforward consequence of this result and \eqref{eq:og4xh}. By \eqref{eq:ufdbdi}, $\xi$ and $\scalarprod{Z}{E_2}_W=w^2\scalarprod{Z}{E_2}$ have the same sign. 
\end{proof}
To proceed further, we need to know the local components of the Cartan tensor. 
\begin{proposition}
\begin{align}
\cartanb{E_2}{E_2}{E_2}{W}&=\frac{3}{2}\scalarprod{X_0}{E_2},\label{eq:car01}\\
\cartanb{E_2}{E_2}{E_3}{W}&=0,\label{eq:7fw56}\\
\cartanb{E_3}{E_3}{E_3}{W}&=0,\label{eq:7v3aszuj}\\
\cartanb{E_3}{E_3}{E_2}{W}&=\frac{1}{2}\scalarprod{X_0}{E_2}.\label{eq:7gzv}
\end{align}
\end{proposition}
\begin{proof}
Equation \eqref{eq:cartan} implies that for arbitrary $U$
\begin{align}\label{eq:snxl}
2\cartanb{E_2}{E_2}{U}{W}&=
3\scalarprod{X_0}{W}\scalarprod{W}{E_2}^2\scalarprod{W}{U}\\
&\quad-2\scalarprod{X_0}{W}\scalarprod{W}{E_2}\scalarprod{E_2}{U}
-\scalarprod{X_0}{W}\scalarprod{W}{U}\scalarprod{E_2}{E_2}\nonumber\\
&\quad-\scalarprod{X_0}{U}\scalarprod{W}{E_2}^2
-2\scalarprod{X_0}{E_2}\scalarprod{W}{U}\scalarprod{W}{E_2}\nonumber\\
&\quad+2\scalarprod{X_0}{E_2}\scalarprod{E_2}{U}
+\scalarprod{X_0}{U}\scalarprod{E_2}{E_2}.\nonumber
\end{align}
Let $U=E_2$. \eqref{eq:8sf6s} gives
\begin{align*}
2\cartanb{E_2}{E_2}{E_2}{W}&=
3\scalarprod{X_0}{W}\scalarprod{W}{E_2}^3
-3\scalarprod{X_0}{W}\scalarprod{W}{E_2}\scalarprod{E_2}{E_2}\\
&\quad-3\scalarprod{X_0}{E_2}\scalarprod{W}{E_2}^2
+3\scalarprod{X_0}{E_2}\scalarprod{E_2}{E_2}\\
&=
3\scalarprod{X_0}{W}\scalarprod{W}{E_2}^3
-3\scalarprod{X_0}{W}\scalarprod{W}{E_2}\scalarprod{E_2}{E_2}\\
&\quad-3\scalarprod{X_0}{E_2}^3
+3\scalarprod{X_0}{E_2}\scalarprod{E_2}{E_2}\\
&=3(1+\scalarprod{X_0}{W})\scalarprod{X_0}{E_2}
(\scalarprod{E_2}{E_2}-\scalarprod{X_0}{E_2}^2)\\
&=3\scalarprod{X_0}{E_2},\quad \text{by \eqref{eq:og4xh}}.
\end{align*}
Thus, \eqref{eq:car01} holds.

Similarly, let $U=E_3$ in \eqref{eq:snxl}.
\begin{align*}
2\cartanb{E_2}{E_2}{E_3}{W}&=
3\scalarprod{X_0}{W}\scalarprod{W}{E_2}^2\scalarprod{W}{E_3}\\
&\quad-2\scalarprod{X_0}{W}\scalarprod{W}{E_2}\scalarprod{E_2}{E_3}
-\scalarprod{X_0}{W}\scalarprod{W}{E_3}\scalarprod{E_2}{E_2}\\
&\quad-\scalarprod{X_0}{E_3}\scalarprod{W}{E_2}^2
-2\scalarprod{X_0}{E_2}\scalarprod{W}{E_3}\scalarprod{W}{E_2}\\
&\quad 2\scalarprod{X_0}{E_2}\scalarprod{E_2}{E_3}+
\scalarprod{X_0}{E_3}\scalarprod{E_2}{E_2}\\
&= 2\scalarprod{X_0}{W}\scalarprod{X_0}{E_2}\scalarprod{E_2}{E_3}+
2\scalarprod{X_0}{E_2}\scalarprod{E_2}{E_3},
\quad\text{by \eqref{eq:8sf6s}}\\
&=0,\quad\text{by \eqref{eq:762f7s}},
\end{align*}
which provides \eqref{eq:7fw56}.

Again, from \eqref{eq:cartan} we have
\begin{align}\label{eq:kf5d}
2\cartanb{E_3}{E_3}{U}{W}&=
3\scalarprod{X_0}{W}\scalarprod{W}{E_3}^2\scalarprod{W}{U}\\
&\quad-\scalarprod{X_0}{W}\scalarprod{E_3}{U}\scalarprod{E_3}{W}-
\scalarprod{X_0}{W}\scalarprod{E_3}{E_3}\scalarprod{U}{W}\nonumber\\
&\quad-\scalarprod{X_0}{U}\scalarprod{W}{E_3}^2
-2\scalarprod{X_0}{E_3}\scalarprod{W}{U}\scalarprod{W}{E_3}\nonumber\\
&\quad+\scalarprod{X_0}{U}\scalarprod{E_3}{E_3}
+2\scalarprod{X_0}{E_3}\scalarprod{U}{E_3}.\nonumber
\end{align}
Set $U=E_3$. 
\begin{align*}
2\cartanb{E_3}{E_3}{U}{W}&=
3\scalarprod{X_0}{W}\scalarprod{W}{E_3}^3
-3\scalarprod{X_0}{W}\scalarprod{E_3}{E_3}\scalarprod{E_3}{W}\\
&\quad-3\scalarprod{X_0}{E_3}\scalarprod{W}{E_3}^2
+3\scalarprod{X_0}{E_3}\scalarprod{E_3}{E_3}.
\end{align*}
By Corollary~\ref{cor:7f65dgb} we have \eqref{eq:7v3aszuj}.

Finally, substitute $U=E_2$ in \eqref{eq:kf5d}.
\begin{align*}
2\cartanb{E_3}{E_3}{E_2}{W}&
=3\scalarprod{X_0}{W}\scalarprod{W}{E_3}^2\scalarprod{W}{E_2}\\
&\quad -2\scalarprod{X_0}{W}\scalarprod{E_3}{E_2}\scalarprod{E_2}{W}
-\scalarprod{X_0}{W}\scalarprod{E_3}{E_3}\scalarprod{E_2}{W}\\
&\quad-\scalarprod{X_0}{E_2}\scalarprod{W}{E_3}^2
-2\scalarprod{X_0}{E_3}\scalarprod{W}{E_2}\scalarprod{W}{E_3}\\
&\quad +\scalarprod{X_0}{E_2}\scalarprod{E_3}{E_3}
+2\scalarprod{X_0}{E_3}\scalarprod{E_2}{E_3}\\
&=\scalarprod{E_3}{E_3}\scalarprod{X_0}{E_2}
\left(
1+\scalarprod{X_0}{W}
\right),\quad\text{by Corollary~\ref{cor:7f65dgb}}\\
&=\scalarprod{X_0}{E_2},\quad\text{by \eqref{eq:8v45vb}}.\qedhere
\end{align*}
\end{proof}

Now, we give an explicit formula for $E_3$.
Since we are in a three-dimensional setting, $E_3$ can be constructed as a cross product of $E_1$ and $ E_2$, where cross product is determined by the scalar product $\scalarprod{}{}_W$. However, $\scalarprod{E_3}{W}=\scalarprod{E_3}{E_2}=0$, thus $E_3$ should be parallel to $W\times E_2$ where the cross product $\times$ now refers to the scalar product $\scalarprod{}{}$. In fact, $E_3=\pm W\times E_2$, as we see from the following statement.
\begin{proposition}
$\|W\times E_2\|_W=1$ where $\times $ denotes the cross product w.r.t.\ the scalar product $\scalarprod{}{}$.
\end{proposition}
\begin{proof}
Substituting into \eqref{eq:scalar} we have
\[
\|W\times E_2\|_W^2=(1+\scalarprod{X_0}{W})\|W\times E_2\|^2.
\]
Now we calculate $\|W\times E_2\|$ separately. From \eqref{eq:8sf6s} and \eqref{eq:og4xh} we find that
\begin{align*}
\|W\times E_2\|^2&=\|W\|^2\cdot\|E_2\|^2-\scalarprod{W}{E_2}^2\\
&=\scalarprod{E_2}{E_2}-\scalarprod{X_0}{E_2}^2=\frac{1}{1+\scalarprod{X_0}{W}}.
\qedhere
\end{align*}
\end{proof}

We fix now the the direction of $E_3$ by $E_3=W\times E_2$.
\begin{lemma}\label{lemma:9cnv4n}
With the notations and hypotheses above,
\begin{align}
[E_3,E_1]&=\scalarprod{E_2}{Z}Z\label{eq:hvu7tvh}\\
[E_2,E_3]&=\left(
\scalarprod{E_2}{E_2}\scalarprod{W}{Z}+\xi\scalarprod{Z}{E_2}^2
\right)Z.\label{eq:hvu7tvh2}
\end{align}
\end{lemma}
\begin{proof}
In view of the relation $[U,V]=\scalarprod{U\times V}{Z}Z$, we obtain from the triple product identity that
\begin{align*}
[E_3,E_1]&=\scalarprod{(W\times E_2)\times E_1}{Z}Z=
(\scalarprod{W}{E_1}\scalarprod{E_2}{Z}-\scalarprod{E_2}{E_1}\scalarprod{W}{Z})Z\\
&=\frac{1}{w}\scalarprod{E_2}{Z}(1+\scalarprod{W}{X_0})Z=\scalarprod{E_2}{Z}Z.
\end{align*}
A similar computation shows \eqref{eq:hvu7tvh2}.
\end{proof}
\begin{proposition}\label{prop:7xf3g}
If the Randers-type Minkowski functional on the three-dimensional Heisenberg algebra is determined by $X_0=\xi Z\in\mathcal Z$ ($0<\xi<1$), the reference vector $W\notin \mathcal{Z}$ and $\|W\|=1$ then 
\begin{equation}\label{eq:if33}
\nabla^W_{E_1}W=f_1E_3,\quad
\nabla^W_{E_2}W=f_2E_3,\quad
\nabla^W_{E_3}W=f_3E_2
\end{equation} 
where $f_1$, $f_2$, $f_3$ are functions of $w=\|W\|_W$. Explicitly, we have
\begin{align*}
f_1&=\scalarprod{[E_3,E_1]}{W}_W,\\
f_2&=\frac{1}{2}\big(
\scalarprod{[E_3,W]}{E_2}_W+\scalarprod{[E_3,E_2]}{W}_W
-\scalarprod{[E_3,W]}{W}_W\scalarprod{X_0}{E_2}
\big),\\
f_3&=\frac{1}{2}\big(
\scalarprod{[E_3,W]}{E_2}_W+\scalarprod{[E_2,E_3]}{W}_W
-\scalarprod{[E_3,W]}{W}_W\scalarprod{X_0}{E_2}
\big).
\end{align*}
\end{proposition}
\begin{proof}
We follow the `local strategy'.
From \eqref{eq:CR} we determine the coordinates of $\nabla^W_WW$ w.r.t.\ the Berwald--Mo\'or frame. Since
\[
2\scalarprod{\nabla_W^WW}{E_i}_W=-\scalarprod{[W,E_i]}{W}_W+
\scalarprod{[E_i,W]}{W}_W=
-2\scalarprod{[W,E_i]}{W}_W,
\]
it follows that
\begin{gather*}
\scalarprod{\nabla^W_WW}{E_1}_W=0,\quad
\scalarprod{\nabla^W_WW}{E_2}_W=0,\quad
\scalarprod{\nabla^W_WW}{E_3}_W=\scalarprod{[E_3,W]}{W}_W,
\end{gather*}
which formulae lead to $\nabla^W_WW=\scalarprod{[E_3,W]}{W}_WE_3$, thus 
\[
\nabla^W_{E_1}{E_1}=\scalarprod{[E_3,E_1]}{E_1}_WE_3\text{ and }
\nabla^W_{E_1}{W}=\scalarprod{[E_3,E_1]}{W}_WE_3=f_1E_3.
\]

Similarly, \eqref{eq:CR} yields
\begin{multline*}
2\scalarprod{\nabla^W_{E_2}W}{U}_W=
-\scalarprod{[W,U]}{E_2}_W+\scalarprod{[U,E_2]}{W}_W\\
-2\scalarprod{[E_3,W]}{W}_W\cartanb{E_3}{U}{E_2}{W}.
\end{multline*}
Thus
\[
\scalarprod{\nabla^W_{E_2}W}{E_1}_W=0,\ \scalarprod{\nabla^W_{E_2}W}{E_2}_W=0
\]
and from \eqref{eq:7gzv}
\begin{multline*}
2\scalarprod{\nabla^W_{E_2}W}{E_3}_W=
-\scalarprod{[W,E_3]}{E_2}_W+\scalarprod{[E_3,E_2]}{W}_W\\
-\scalarprod{[E_3,W]}{W}_W\scalarprod{X_0}{E_2}=2f_2,
\end{multline*}
and this implies that $\nabla^W_{E_2}W=f_2E_3$.

Finally,
\begin{multline*}
2\scalarprod{\nabla^W_{E_3}W}{U}_W=
\scalarprod{[E_3,W]}{U}_W-\scalarprod{[W,U]}{E_3}_W+\scalarprod{[U,E_3]}{W}_W\\
-2\scalarprod{[E_3,W]}{W}_W\cartanb{E_3}{U}{E_3}{W},
\end{multline*}
which yields to
\[
\scalarprod{\nabla^W_{E_3}W}{E_1}_W=0,\
\scalarprod{\nabla^W_{E_3}W}{E_3}_W=0
\]
and
\[
\scalarprod{\nabla^W_{E_3}}{E_2}_W=
\scalarprod{[E_3,W]}{E_2}_W+\scalarprod{[E_2,E_3]}{W}_W-
\scalarprod{[E_3,W]}{W}_W\scalarprod{X_0}{E_2}=2f_3,
\]
which gives the last statement of \eqref{eq:if33}.

We show that $f_1$ depends only on $w$. Combining \eqref{eq:hvu7tvh} with \eqref{eq:scalar} we obtain
\[
f_1=\scalarprod{E_2}{Z}\scalarprod{Z}{W}_W
=w(\xi+\scalarprod{Z}{W})\scalarprod{E_2}{Z}.
\]
In Corollary~\ref{cor:76rvg} we computed $\scalarprod{E_2}{Z}$ directly from $w$, moreover we have \[\scalarprod{Z}{W}=\frac{w-1}{\xi}.\] 
Thus
\[
f_1=\sqrt{\frac{\xi^2-(w-1)^2}{\xi^2w^3}}w\left(\xi+\frac{w-1}{\xi}\right)=
\frac{1}{\xi}\sqrt{\frac{\xi^2-(w-1)^2}{w}}\left(\xi+\frac{w-1}{\xi}\right).
\]

Argument similar to that of the previous statement shows that 
\[
\scalarprod{[E_3,W]}{E_2}_W=w^3\scalarprod{Z}{E_2}^2,
\] 
and
\[
\scalarprod{[E_2,E_3]}{W}_W
-\scalarprod{[E_3,W]}{W}_W\scalarprod{X_0}{E_2}=
\scalarprod{Z}{W}^2+\scalarprod{X_0}{W}.
\]
Combining these relations yields $f_3=\frac{1}{2}w$.

Finally, 
\[
f_2=f_3+\scalarprod{[E_3,E_2]}{W}_W=\frac{w}{2}-
\frac{\xi^2+w-1}{w}\left(
1+\frac{w-1}{\xi^2}
\right).\qedhere
\]
\end{proof}
\begin{proposition}\label{prop:7xf37cg}
If the Randers-type Minkowski functional on the three-dimensional Heisenberg algebra is determined by $X_0=\xi Z\in\mathcal Z$ ($0<\xi<1$), the reference vector $W\notin \mathcal{Z}$ and $\|W\|=1$ then 
the local components of the Chern--Rund connection $\nabla^W$ w.r.t.\ Berwald-Mo\'or frame are
\begin{align*}
&\nabla^W_{E_1}E_1=\frac{f_1}{w}E_3,\quad
\nabla^W_{E_1}E_2=\nabla^W_{E_2}E_1=\frac{f_2}{w}E_3,\quad
\nabla^W_{E_2}E_2=f_4E_3,\\
&\nabla^W_{E_3}E_1=\nabla^W_{E_1}{E_3}+[E_3,E_1]=\frac{f_3}{w}E_2,\\
&\nabla^W_{E_2}E_3=\nabla^W_{E_3}E_2+[E_2,E_3]=-\frac{f_2}{w}E_1+f_5E_2,\\
&\nabla^W_{E_3}E_3=-\frac{f_3}{2}\scalarprod{X_0}{E_2}E_3,
\end{align*} 
where $w=\|W\|_W$, $f_1$, $f_2$, $f_3$ are defined in Proposition~\ref{prop:7xf3g}, and
\begin{align*}
f_4&=-\frac{f_1}{w}-f_2\scalarprod{X_0}{E_2}+\frac{3w}{4}\scalarprod{X_0}{E_2}\\
f_5&=\scalarprod{[E_2,E_3]}{E_2}_W-\frac{3}{2}f_3\scalarprod{X_0}{E_2}.
\end{align*}
\end{proposition}

\begin{proof}
The proof is similar to that given in the Proposition~\ref{prop:7xf3g}.
\end{proof}
\section{The flag curvature of left invariant Randers metrics on 3-dimensional Heisenberg group}
The geometry of any Lie group with left invariant Riemannian metric reflects strongly the algebraic structure of the corresponding Lie algebra. Papers e.g.~by J.~Wolf, J.~Milnor 
(\cite{MR0425012,MR0162206}) serve many evidence for this statement. As an example we recall P.~Eberlein's result for 2-step nilpotent groups with left invariant Riemannian metric.
\begin{theoremm}[Eberlein, \cite{MR1296558}]
$\Pi=\operatorname{span}(X,Y)\subseteq \mathcal N$, 
where $(X,Y)$ is orthonormal pair.  Let $K(\Pi)=K(X,Y)$ is the sectional curvature map. Then
  \begin{align*}
  &K(X,Y)=-\frac{3}{4}\| [X,Y]\|^2, \quad X,Y\in\centrumort\\
  &K(X,Z)=\frac{1}{4}\|j(Z)X\|^2, \quad
    X\in\centrumort,Z\in\centrum\\
  &K(Z,Z^*)=0,\quad Z,Z^*\in\centrum
  \end{align*}
where $\centrum$ is the center of the Lie algebra, $\centrumort$ is the orthogonal complement of the center w.r.t.\ the left invariant Riemannian metric and $j(Z)\colon\centrumort\to\centrumort$ is defined by
$\scalarprod{j(Z)X}{Y}=\scalarprod{[X,Y]}{Z}$ for all $X,Y\in\centrumort$.
  \end{theoremm}
The purpose of this section is to generalise this result to left invariant Randers metrics on the 3-dimensional Heisenberg group. We note that S.~Deng and Z.~Hu have recently obtained some remarkable results for curvatures of homogeneous Randers metrics (\cite{MR3046307,MR3004458,MR2892934}).

The flag curvature of the Finsler manifold $N$ is determined by 
a basepoint $p\in N$, the flagpole $W\in T_pN$ and an edge (transverse vector) $U\in T_pN$
by the formula
\begin{equation}\label{eq:flag}
K(\Pi, W)=K(\Pi)=\frac{\scalarprod{R(U,W)W}{U}_W}%
{\scalarprod{U}{U}_W\scalarprod{W}{W}_W-\scalarprod{U}{W}_W^2}
\end{equation}
where $\Pi=\operatorname{span}(U,W)$ and $R$ is the affine curvature tensor of the Chern-Rund connection (see e.g.\ \cite[Section 3.9.]{MR1747675}).

\begin{theorem}
If the Randers-type Minkowski functional on the three-dimensional Heisenberg algebra  ($=\operatorname{span}(X,Y,Z)$, as in definition \ref{def:3wjswg}) is determined by $X_0=\xi Z\in\mathcal Z$ ($0<\xi<1$) and $W=Z$, then the flag curvature of the Chern-Rund connection is
\[
K(\Pi)=\frac{1}{4}\quad
\text{for all } U\in\Span(X,Y).
\]
\end{theorem}
\begin{proof}
Let $U=\alpha X+\beta Y$. From relations in Proposition \ref{prop2542} we get
\[
R(U,W)W=\frac{\alpha(\xi+1)^2}{4}X+\frac{\beta(\xi+1)^2}{4}Y,
\]
and an easy calculation gives the statement.
\end{proof}

In what follows $W\in\Span(X,Y)$ and we get special case of Proposition \ref{prop:7xf37cg}.
\begin{proposition}\label{prop:5r6}
(With the notations from Proposition \ref{prop:7xf37cg}.) If $W\in\Span(X,Y)$ then $\Span(X,Y)=\Span(E_1,E_3)$ and $E_2=-\xi E_1+Z$. Moreover,
$f_1=\xi$, $f_2=\frac{1}{2}-\xi^2$, $f_4=\xi^3-\frac{3\xi}{4}$, $f_5=\frac{\xi}{4}$, and the local components of the Chern-Rund connection w.r.t.\ Berwald-Mo\'or frame are
\begin{align*}
&\nabla^W_{E_1}E_1=\xi E_3,\quad
\nabla^W_{E_1}E_2=\nabla^W_{E_2}E_1=\left(\frac{1}{2}-\xi^2\right)E_3,\quad
\nabla^W_{E_2}E_2=\xi\left(\xi^2-\frac{3}{4}\right)E_3,\\
&\nabla^W_{E_3}E_1=\frac{1}{2}E_2, \
\nabla^W_{E_1}{E_3}=-\xi E_1-\frac{1}{2}E_2,\\
&\nabla^W_{E_2}E_3=\left(\xi^2-\frac{1}{2}\right)E_1+\frac{1}{4}\xi E_2, \
\nabla^W_{E_3}E_2=-\frac{1}{2}E_1-\frac{3}{4}\xi E_2,\\
&\nabla^W_{E_3}E_3=-\frac{1}{4}\xi E_3.
\end{align*} 
\end{proposition}
\begin{proof}
$E_3$ is always orthogonal to the centrum (in both senses, i.e.\ with respect to the euclidean scalar product $\scalarprod{}{}$ and osculating scalar product $\scalarprod{}{}_W$, see Corollary \ref{cor:7f65dgb}), which means that $E_3\in\Span(X,Y)$.

Substituting into \eqref{eq:scalar} we have
\begin{align}%
\scalarprod{E_1}{Z}_W&=\xi\label{eq:726w7}\\
\scalarprod{W}{W}_W&=1\\
\scalarprod{Z}{Z}_W&=1+\xi^2\label{eq:ztctr}.
\end{align}
It follows that $E_1=W$ and
\begin{align*}
\scalarprod{E_1}{-\xi E_1+Z}_W=0\\
\scalarprod{-\xi E_1+Z}{-\xi E_1+Z}_W=1\\
\scalarprod{E_3}{-\xi E_1+Z}_W=0.
\end{align*}
Thus $E_2=-\xi E_1+Z$ or $E_2=\xi E_1-Z$. From Corollary \ref{cor:76rvg} we know that 
$\operatorname{sgn}\scalarprod{Z}{E_2}=\operatorname{sgn} \xi>0$ and this fact implies $E_2=-\xi E_1+Z$.
Now, the statements are simple consequences of Proposition \ref{prop:7xf37cg}.
\end{proof}

\begin{theorem}
If the Randers-type Minkowski functional on the three-dimensional Heisenberg algebra  is determined by $X_0=\xi Z\in\mathcal Z$ ($0<\xi<1$) and
$W\in\Span(X,Y)$, then the the flag curvature of the Chern-Rund connection is
\begin{enumerate}[a)]
\item $K(\Pi)=\frac{\xi^2-3}{4}<0$ \label{ita}
for all $U\in\operatorname{span}(X,Y)$ 
\item $K(\Pi)=\frac{1-\xi^2}{4}>0$ for $U=Z$.\label{itb}
\end{enumerate}
\end{theorem}
\begin{proof}
Let $U=\alpha E_1+\beta E_3$. From Proposition~\ref{prop:5r6} we have
\[
R(U,W)W=\beta\frac{\xi^2-3}{4}E_3,
\]
so
\[
\scalarprod{R(U,W)W}{U}_W=
\beta \frac{\xi^2-3}{4}\scalarprod{E_3}{\alpha E_1+\beta E_3}_W=
\beta^2 \frac{\xi^2-3}{4}.
\]
Applying \eqref{eq:726w7}--\eqref{eq:ztctr}, the denominator of \eqref{eq:flag} is $\beta^2$ and statement \emph{\ref{ita})} holds.

Let $U=Z$. A simple substitution into Proposition~\ref{prop:5r6} gives
\[
R(Z,W)W=\frac{1-\xi^2}{4}E_2,
\]
and we have
\[
\scalarprod{R(Z,W)W}{Z}_W=\frac{1-\xi^2}{4}\scalarprod{E_2}{\xi W+E_2}_W=
\frac{1-\xi^2}{4}.
\]
Moreover, from \eqref{eq:726w7}--\eqref{eq:ztctr} we get
\[
\scalarprod{U}{U}_W\scalarprod{W}{W}_W-\scalarprod{U}{W}_W^2=1,
\]
then we obtain statement \emph{\ref{itb})}.
\end{proof}
\def\polhk#1{\setbox0=\hbox{#1}{\ooalign{\hidewidth
  \lower1.5ex\hbox{`}\hidewidth\crcr\unhbox0}}}

\end{document}